\documentclass[11pt]{amsart}

\usepackage{amsmath, amsfonts, amssymb, amsthm}
\usepackage{color}
\usepackage[left = 1.5cm, right = 1.5cm, bottom = 2.5cm]{geometry}
\usepackage{enumitem}
\usepackage{tikz}
\usepackage{mathtools}
\usepackage{float}

\newcommand*\mR{\mathbb{R}}

\newcommand*\ssp{\mathcal{X}^{s,p}(D)}
\newcommand*\wsp{W^{s,p}(D)}

\newcommand*\wosp{W_0^{s,p}(D)}
\newcommand*\sosp{\mathcal{X}_0^{s,p}(D)}

\newcommand*\wtu{\widetilde{u}}

\theoremstyle{plain}
\newtheorem{proposition}{Proposition}[section]
\newtheorem{lemma}[proposition]{Lemma}
\newtheorem{theorem}[proposition]{Theorem}

\theoremstyle{definition}
\newtheorem{definition}[proposition]{Definition}

\newtheorem{remark}[proposition]{Remark}
\numberwithin{equation}{section}
\title[Fractional Korn's inequality]{Fractional Korn's inequality on subsets of the Euclidean space}
\author{Artur Rutkowski}
\address{Faculty of Pure and Applied Mathematics,
	Wroc\l aw University of Science and Technology,
	Wyb. Wyspia\'nskiego 27, 50-370 Wroc\l aw, Poland}
\email{artur.rutkowski@pwr.edu.pl}
\thanks{\hspace{-11pt}2020 \textit{Mathematics Subject Classification}. 26D10, 46E35, 46E40.}

\begin{document}
\maketitle
\begin{abstract}
	We give a new, simpler proof of the fractional Korn's inequality for subsets of $\mR^d$. We also show a framework for obtaining Korn's inequality directly from the appropriate Hardy-type inequality.
\end{abstract}
\section{Introduction}
Let $D$ be a bounded open subset of $\mR^d$, $d>1$, and let $p\in (1,\infty)$. For $x\in \mR^d$ we will use the notation $x = (x',x_d)$ with $x'\in \mR^{d-1}$, $x_d\in \mR$. Whenever we mention a vector field on $D\subseteq \mR^d$, we mean a measurable mapping from $D$ into $\mR^d$. The space $L^p(D)$ consists of all the vector fields $u$ for which the norm
$\|u\|_{L^p(D)} := (\int_D |u(x)|^p \, dx)^{1/p}$ is finite.

We define the fractional Sobolev space of the vector fields as follows:
$$\wsp = \bigg\{u\in L^p(D): |u|_{\wsp}^p  := \int_D\int_D\frac{\big|u(x) - u(y)\big|^p}{|x-y|^{d+sp}} \, dx \, dy < \infty\bigg\}.$$
$\wsp$ is endowed with the norm given by the formula $\|u\|_{\wsp} := (\|u\|_{L^p(D)}^p + |u|_{\wsp}^p)^{1/p}.$ We also introduce the Sobolev space with the projected difference quotient:
$$\ssp = \bigg\{u\in L^p(D): |u|^p_{\ssp} := \int_D\int_D \frac{\big|(u(x) - u(y))\frac{(x-y)}{|x-y|}\big|^p}{|x-y|^{d+sp}} \, dx\, dy < \infty \bigg\}.$$
We equip $\ssp$ with the norm $\|u\|_{\ssp} := (\|u\|_{L^p(D)}^p + |u|^p_{\ssp})^{1/p}$. Furthermore, we define the spaces $\wosp$ and $\sosp$ as the closures of $(C^1_c(D))^d$ in the norms $\|\cdot\|_{\wsp}$ and $\|\cdot\|_{\ssp}$, respectively.

Obviously, $\|u\|_{\ssp} \leq \|u\|_{\wsp}$. Our main goal here is to establish a reverse inequality with a multiplicative constant on the left-hand side, which is known as the fractional Korn's inequality.
\begin{theorem}\label{th:main}
	Assume that $p>1$, $s\in (0,1)$, and $sp\neq 1$. If $D$ is a bounded $C^1$ open set or a bounded Lipschitz set with sufficiently small Lipschitz constant, then there exists $C\geq 1$ such that for every $u\in \sosp$ we have
	$$C\|u\|_{\ssp} \geq \|u\|_{\wsp}.$$
	In particular, $\sosp = \wosp$.
\end{theorem}
This result was obtained very recently by Mengesha and Scott \cite{2020arXiv201112407M} with the use of a complicated extension operator intrinsic to studying projected seminorms. We present a significantly shorter proof which omits the extension: we first obtain the inequality for the epigraphs in Theorem \ref{th:goodlip} by using the result of Mengesha and Scott \cite[Theorem~1.1]{MR3959431} for the half-space together with an appropriate change of variables, and then we apply an argument via the partition of unity. 

In Section~\ref{sec:Hardy} we show that the Korn's inequality for the vector fields of the class $(C^1_c(D))^d$ can be obtained directly from the appropriate Hardy-type inequality for $D$ with the use of an operator which extends the vector field by 0 to the whole space. This in particular yields a simpler proof of the Korn's inequality for the half-space than the original one due to Mengesha \cite{MR4017780}. It may also facilitate the proofs for more general sets $D$ in the future.

For applications, open problems, and a wider context concerning the fractional Korn's inequality we refer to the aforementioned works of Mengesha and Scott. We~remark that the arguments below were obtained independently of the ones in \cite{2020arXiv201112407M}.

\subsection*{Acknowledgements} I thank Bart\l{}omiej Dyda for helpful discussions and remarks to the manuscript. Research was partially supported by the Faculty of Pure and Applied Mathematics, Wroc\l{}aw University of Science and Technology, 049U/0052/19.

\section{Preliminaries and auxiliary results}
Let $f\colon \mR^{d-1} \to \mR$ be a continuous function. The open set $\{(x',x_d)\in \mR^d: f(x') < x_d\}$ will be called the epigraph of $f$. The epigraph of a Lipschitz (resp. $C^1$) function will be called a Lipschitz (resp. $C^1$) epigraph.

\begin{definition}\label{def:Lip}
	We say that an open set $D\subseteq \mR^d$ is Lipschitz with constant $L>0$ if there exist balls $B_1,\ldots B_n$ with centers belonging to $\partial D$, epigraphs $U_1,\ldots,U_n$ of functions $f_1,\ldots,f_n$ with Lipschitz constant $L$ or better, and rigid motions $R_1,\ldots R_n$, such that the following conditions are satisfied
	\begin{itemize}
		\item $\partial D \subset \bigcup\limits_{i=1}^n B_i$,
		\item $R_i(B_i\cap D) \subset U_i$ and $R_i(B_i\cap\partial D) = R_i(B_i)\cap \partial U_i$.
	\end{itemize}
We say that $D$ is a $C^1$ set if the above conditions are satisfied with the difference that the epigraphs $U_1,\ldots,U_n$ and the functions $f_1,\ldots,f_n$ are $C^1$ instead of Lipschitz.
\end{definition}
We also consider an additional open set (not necessarily a ball) $B_{n+1}$, relatively compact in $D$, such that $D~\subset~ \bigcup\limits_{i=1}^{n+1} B_i$. Note that if $D$ is $C^1$, then the balls and the epigraphs may be so chosen that $D$ is Lipschitz with an arbitrarily small constant $L$.

Throughout the paper we will use certain Lipschitz maps as substitutions in the integration process. Below we establish some basic facts about these transformations.
\begin{lemma}\label{lem:sub}
	Let $U$ and $V$ be open subsets of $\mR^d$ and assume that $T\colon U \to V$ is a bijection such that $T$ and $T^{-1}$ are Lipschitz with constant $K\geq 1$. Then for every non-negative measurable function $u\colon V \to \mR$ we have
	\begin{equation*}
	(1/K)^d\int_V u(x)\, dx\leq\int_U u(Tx) \, dx \leq K^d \int_V u(x) \, dx.
	\end{equation*}
\end{lemma}
\begin{proof}
	This fact follows conveniently from the result of Haj\l{}asz \cite[Appendix]{MR1201446}, see also Rado and Reichelderfer \cite[V.2.3]{MR0079620}. To verify the validity of the constants we first claim that $J_T$ --- the Jacobian of $T$ satisfies $|J_T| \leq K^d$ almost everywhere in $U$. Indeed, let $x_0\in U$ and $r>0$ satisfy $B(x_0,r)\subset U$. If we take $f = T$ and $u = \textbf{1}_{T B(x_0,r)}$ in \cite{MR1201446}, then we get that
	$$\int_{B(x_0,r)} |J_T(x)|\, dx = \int_{T B(x_0,r)}\, dy.$$
	Thus,
	\begin{equation}\label{eq:Lebesgue}
	\frac 1{|B(x_0,r)|}\int_{B(x_0,r)} |J_T(x)|\, dx = \frac {|T B(x_0,r)|}{|B(x_0,r)|}.
	\end{equation}
	The limit $r\to 0^+$ on the left-hand side of \eqref{eq:Lebesgue} exists and equals $J_T(x_0)$ for almost every $x_0\in U$ by the Lebesgue differentiation theorem. Furthermore, we have $TB(x_0,r)\subseteq B(T(x_0),Kr)$, hence the right-hand side of \eqref{eq:Lebesgue} is bounded from above by $K^d$. This proves the claim that $|J_T(x)|\leq K^d$ for almost every $x\in U$. Similarly we show that $|J_{T^{-1}}| \leq K^d$ almost everywhere in $V$. Thus, the lemma follows from \cite{MR1201446} and the formula $|J_{T^{-1}}(Tx)| = |J_T(x)|^{-1}$.
\end{proof}

We will commonly map the epigraph of a Lipschitz function $f\colon \mR^{d-1} \to \mR$ to the half-space $\mR^d_+$ as follows:
$$T(x',x_d) = (x',x_d - f(x')),\qquad x'\in\mR^{d-1},\ x_d > f(x').$$
Clearly this is a bijection with the inverse
$$T^{-1}(x',x_d) = (x', x_d + f(x')), \qquad x'\in \mR^{d-1},\ x_d>0.$$
\begin{lemma}\label{lem:lambdadist}
	If $f$ is Lipschitz with constant $L$, then for every $x$ and $y$ in the epigraph of $f$ we have
	\begin{equation}\label{eq:bilip}C(L)^{-1}|x-y| \leq |T x - T y| \leq C(L)|x-y|,\end{equation}
	where $C(L) = \sqrt{1 + L(L+1)}$. In particular, $C(L) \to 1$ as $L\to 0^+$.
\end{lemma}
\begin{proof}
	Let $x$ and $y$ belong to the epigraph of $f$. We may reduce the problem to the planar geometry. Let $a = |x'-y'|$, $b = |x_d - y_d|$, $c = |x-y|$, $c' = |Tx - Ty|$, $d = |f(x') - f(y')|$. Note that $d\leq La$.
	\begin{figure}[H]
		\begin{tikzpicture}
			\filldraw (0,0) circle(1pt) node[above left]{$x$};
			\draw (0,0)--(3,0)--(3,-1)--(0,0);
			\filldraw (3,-1) circle(1pt) node[below right]{$y$};
			\draw (1.5,0) node[above] {$a$};
			\draw (3,-0.5) node[right] {$b$};
			\draw (1.5,-0.5) node[below left] {$c$};
			\filldraw (6,0) circle(1pt) node[above left]{$Tx$};
			\draw (6,0)--(9,0)--(9,-2.5)--(6,0);
			\filldraw (9,-2.5) circle(1pt) node[below right]{$Ty$};
			\draw (7.5,0) node[above]{$a$};
			\draw (9,-1.25) node[right] {$b+d$};
			\draw (7.5,-1.25) node[below left] {$c'$};
		\end{tikzpicture}
		\caption{The picture presents the \textit{pessimistic} variant with $(x_d - y_d)(f(x') - f(y'))> 0$.}
		\label{fig:1}
	\end{figure}
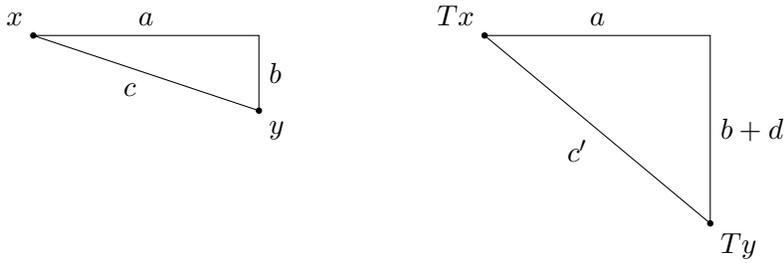
	We will work with the context given in Figure \ref{fig:1}. Assuming that $x\ne y$, we have
	\begin{align}\label{eq:stala}
	\frac {(c')^2}{c^2} = 1 + \frac{2bd + d^2}{a^2+b^2} \leq 1 + \frac{2Lba + L^2a^2}{a^2+b^2} = 1 + L \frac{La^2 + 2ab}{a^2 + b^2} \leq 1 + L\frac{(L+1)a^2 + b^2}{a^2+b^2} \leq 1 + L(L+1).
	\end{align}
	This gives the right-hand side part of \eqref{eq:bilip}. A similar argument may be used with $T^{-1}$ in place of $T$, giving the left-hand side of \eqref{eq:bilip}.
\end{proof}
\begin{lemma}\label{lem:unit}
		Assume that $D$ is bounded and that $\psi\in C^\infty_b(D)$. Then $|u\psi|_{\ssp} \lesssim\|u\|_{\ssp}$. An analogous result holds with $W^{s,p}$ in place of $\mathcal{X}^{s,p}$.
\end{lemma}
\begin{proof}
	We have
	\begin{align*}
	\int_D\int_D\frac{\big| (u\psi(x) - u\psi(y))\frac{(x-y)}{|x-y|}\big|^p}{|x-y|^{d+sp}}\, dx\, dy &\lesssim \int_D\int_D \frac{\big|(\psi(x) - \psi(y))u(x)\frac{(x-y)}{|x-y|}\big|^p}{|x-y|^{d+sp}}\, dx\, dy\\
	&+\int_D\int_D \frac{\big| \psi(y)(u(x) - u(y))\frac{(x-y)}{|x-y|}\big|^p}{|x-y|^{d+sp}}\, dx\, dy.
	\end{align*}
	The latter integral is smaller than $\|\psi\|_{L^\infty(D)}^p|u|_{\ssp}^p$. For the former we use the fact that $|\psi(x) - \psi(y)| \lesssim |x-y|$ to get that it does not exceed $c\|u\|_{L^p(D)}^p$. The proof for $W^{s,p}$ is identical.
\end{proof}
Let $B_\delta = \{x \in B: d(x,\partial B) > \delta\}$. In the subsequent section we will apply an argument using a partition of unity subordinate to $B_1,\ldots,B_n,B_{n+1}$ from Definition~\ref{def:Lip}. This in particular will require extending vector fields given on $B_i\cap D$ and supported in $\overline{(B_i)_\delta\cap D}$ for some fixed $\delta > 0$, to a rotated epigraph ($1\leq i\leq n$) or to the whole of $\mR^d$ ($i=n+1$). The following result enables us to perform such operations.
\begin{lemma}\label{lem:0ext}
	Let the open sets $B,U\subseteq \mR^d$ satisfy $U\cap B_\delta \ne \emptyset$ and $U\setminus B \ne \emptyset$. Assume that $u\in \mathcal{X}^{s,p}(U\cap B)$ has support contained in $\overline{U\cap B_\delta}$ for fixed $\delta > 0$. If we let $\wtu = u$ in $U\cap B$ and $\wtu = 0$ in $U\setminus B$, then $\|\wtu\|_{\mathcal{X}^{s,p}(U)} \lesssim \|u\|_{\mathcal{X}^{s,p}(U\cap B)}.$ Analogous result holds with $W^{s,p}$ in place of $\mathcal{X}^{s,p}$.
\end{lemma}
\begin{proof}
	Obviously, it suffices to estimate $|\wtu|_{\mathcal{X}^{s,p}(U)}$. Since $\wtu = 0$ on $B_\delta^c$ we have
	\begin{align*}
	\int_{U}\int_{U} \frac{\big|(\wtu(x) - \wtu(y))\frac{(x-y)}{|x-y|}\big|^p}{|x-y|^{d+sp}} \, dx \, dy &\leq |u|_{\mathcal{X}^{s,p}(U\cap B)}^p + 2\int_{U\cap B_\delta}|u(x)|^p\int_{U\setminus B} \frac{1}{|x-y|^{d+sp}} \, dy \, dx\\
	&\leq |u|_{\mathcal{X}^{s,p}(U\cap B)}^p + c(\delta)\|u\|_{L^p(U\cap B)}^p\lesssim \|u\|_{\mathcal{X}^{s,p}(U\cap B)}^p.
	\end{align*}
	The proof for $W^{s,p}$ is identical.
\end{proof}

\section{Proof of the Korn's inequality}
We will show that the Korn's inequality holds for the epigraphs with sufficiently small Lipschitz constant and then use this fact to establish Theorem \ref{th:main}.

\begin{theorem}\label{th:goodlip}
	Assume that $sp\neq 1$ and that $D$ is the epigraph of a Lipschitz function $f$ with sufficiently small Lipschitz constant $L$. Then there exists $C\geq 1$ such that for every $u\in\sosp$ we have
	$$C\|u\|_{\ssp} \geq \|u\|_{\wsp}.$$
	Consequently, $\sosp = \wosp$.
\end{theorem}
\begin{proof}
	Following the approach of Nitsche \cite[Remark 3]{MR631678} we will show that there exist $c_1 = c_1(L)$ and $c_2 = c_2(L)$, such that
	\begin{equation}\label{eq:subtr}
	|u|_{\wsp}^p \leq c_1|u|_{\ssp}^p + c_2|u|_{\wsp}^p.
	\end{equation}
	We will propose an explicit form of $c_1$ and $c_2$ so that it will be obvious that for sufficiently small $L$ we have $c_2 < 1$ and the statement will follow by subtracting $c_2|u|_{\wsp}^p$.
	
	Let $u\in (C^1_c(D))^d\subseteq \wosp$. If we substitute $(w',w_d) = (x',x_d - f(x'))$ and $(z',z_d) = (y',y_d - f(y'))$, then by Lemmas  \ref{lem:sub} and \ref{lem:lambdadist} (see the latter for the definition of $C(L)$) we get
	\begin{align}
	|u|_{\wsp}^p &\leq C(L)^{2d} \int_{\mR^d_+}\int_{\mR^d_+}\frac{|u(w',w_d + f(w')) - u(z',z_d + f(z'))|^p}{|(w',w_d + f(w')) - (z', z_d + f(z'))|^{d+sp}} \, dz \, dw\nonumber\\
	&\leq C(L)^{3d+sp}\int_{\mR^d_+}\int_{\mR^d_+} \frac{|u(w',w_d + f(w')) - u(z',z_d + f(z'))|^p}{|w-z|^{d+sp}} \, dz \, dw.\label{eq:uv}
	\end{align}
	Let $v(w',w_d) = u(w',w_d + f(w'))$. Note that the above inequalities are in fact comparisons, in particular the double integral in \eqref{eq:uv} is finite, which means that $v \in W_0^{s,p}(\mR^d_+)$. Now, if we let $C_K$ be the constant in the Korn's inequality of Mengesha and Scott \cite[Theorem~1.1]{MR3959431}, and $u^d$ --- the $d$-th coordinate of $u$, then we can estimate the double integral in \eqref{eq:uv} as follows:
	\begin{align}
	&\int_{\mR^d_+}\int_{\mR^d_+} \frac{|v(w) - v(z)|^p}{|w-z|^{d+sp}} \, dz \, dw\nonumber\\
	&\leq C_K \int_{\mR^d_+}\int_{\mR^d_+} \frac{\big|(v(w) - v(z))\frac{(w-z)}{|w-z|}\big|^p}{|w-z|^{d+sp}} \, dz \, dw\nonumber\\
	&= C_K\int_{\mR^d_+}\int_{\mR^d_+} \frac{\big|(u(w',w_d + f(w')) - u(z',z_d + f(z')))\frac{(w-z)}{|w-z|}\big|^p}{|w-z|^{d+sp}} \, dz \, dw\nonumber\\
	&\leq 2^{p-1}C_K \int_{\mR^d_+}\int_{\mR^d_+} \frac{\big|(u(w',w_d + f(w')) - u(z',z_d + f(z')))((w',w_d+f(w')) - (z',z_d + f(z')))\big|^p}{|w-z|^{d+sp+p}}\, dz \, dw\label{eq:proj}\\
	&+ 2^{p-1}C_K\int_{\mR^d_+}\int_{\mR^d_+} \frac{\big|(u^d(w',w_d + f(w')) - u^d(z',z_d + f(z')))(f(z') - f(w'))\big|^p}{|w-z|^{d+sp+p}}\, dz \, dw.\label{eq:sob}
	\end{align}
	By going back to the old variables and by using Lemma \ref{lem:lambdadist} once more, we get that \eqref{eq:proj} is estimated from above by
	\begin{align*}
	2^{p-1}C_K C(L)^{3d+sp+p} &\int_{D}\int_{D} \frac{\big|(u(x) - u(y))(x - y)\big|^p}{|x-y|^{d+sp+p}}\, dx \, dy = 2^{p-1}C_K C(L)^{3d+sp+p} |u|_{\ssp}^p.
	\end{align*}
	In \eqref{eq:sob} we also substitute the old variables so that it is estimated from above by
	\begin{align*}
	&2^{p-1}C_KC(L)^{3d+sp+p}\int_{D}\int_{D} \frac{\big|(u^d(x) - u^d(y))(f(y') - f(x'))\big|^p}{|x-y|^{d+sp+p}}\, dx \, dy\\
	\leq\, &2^{p-1}C_KC(L)^{3d+sp+p}L^p\int_{D}\int_{D} \frac{\big|u^d(x) - u^d(y)\big|^p}{|x-y|^{d+sp}}\, dx \, dy\\
	\leq \, &2^{p-1}C_KC(L)^{3d+sp+p}L^p|u|_{\wsp}^p.
	\end{align*}
	Overall, we get \eqref{eq:subtr}:
	\begin{equation*}
	|u|_{\wsp}^p \leq 2^{p-1}C_K C(L)^{6d+2sp+p}|u|_{\ssp}^p + 2^{p-1}C_K C(L)^{6d+2sp+p}L^p|u|_{\wsp}^p.
	\end{equation*}
	Since $C(L)\approx 1$ for small $L$, the second constant can be made arbitrarily small, which yields the Korn's inequality for $u\in (C^1_c(D))^d$. 
	
	Now let $u\in \wosp$. There exists a sequence $u_n\in (C^1_c(D))^d$ such that $\|u_n - u\|_{\wsp}\to 0$ as $n\to\infty$. From this we infer that $\|u_n - u\|_{\ssp} \to 0$ as well, and by using the Korn's inequality for $u_n$ we get that
	$$|u|_{\wsp} \leq |u_n - u|_{\wsp} + |u_n|_{\wsp} \leq |u_n -u|_{\wsp} + C|u_n|_{\ssp}.$$
	By letting $n\to \infty$, we obtain the Korn's inequality for $\wosp$. Since the identity mapping is a bi-Lipschitz homeomorphism from $\wosp$ to $\sosp$ and both are Banach spaces, the former is a closed subspace of the latter. By density, $\wosp = \sosp$.
\end{proof}

\begin{proof}[Proof of Theorem \ref{th:main}]
First, suppose that $u\in \wosp$. We may assume that $D$ is a Lipschitz set with the constant $L$ satisfying the assumptions of Theorem \ref{th:goodlip}. Let $B_1,\ldots,B_n, B_{n+1},\, R_1,\ldots, R_n, \, f_1,\ldots,f_n$, and $U_1,\ldots, U_n$ be as in Definition \ref{def:Lip}, and let $U_{n+1} = \mR^d$ and $R_{n+1} = I$. We consider a smooth partition of unity $\psi_1,\ldots, \psi_{n+1}$ subordinate to $B_1,\ldots, B_{n+1}$, i.e., $0\leq\psi_i\leq 1$, $\textrm{supp}(\psi_i) \subset B_i$, and $\sum \psi_i = 1$ on $D$.

We define $u_i = u\psi_i$, $i=1,\ldots, n+1$. By the triangle inequality we have
\begin{equation*}
|u|_{\wsp} \lesssim \sum\limits_{i=1}^{n+1} |u_i|_{\wsp}.
\end{equation*}
Furthermore, $|u_i|_{\wsp} \lesssim \|u_i\|_{W^{s,p}(B_i\cap D)}$ by Lemma \ref{lem:0ext}.

Now, for $i=1,\ldots,n+1$ we extend $R_i(u_i)(R_i^{-1} (\cdot))$ from $R_i(B_i\cap D)$ to $U_i$ by 0 and we call the resulting vector fields $\wtu_i$. Unlike $\|u\|_{\ssp}$, the norm $\|u\|_{\wsp}$ is invariant under the rotations of $u$, so it is crucial that we rotate $u_1,\ldots,u_n$ at this point, so that they agree with the new coordinate system. Obviously, we have $\|u_i\|_{W^{s,p}(B_i\cap D)} \leq \|\wtu_i\|_{W^{s,p}(U_i)}$ and by Lemma \ref{lem:0ext} the right-hand sides are finite for all $i$, hence $\wtu_i\in W_0^{s,p}(U_i)$. By using Theorem \ref{th:goodlip} for $\wtu_1,\ldots,\wtu_n$ and the Korn's inequality for the whole space \cite[Theorem~1.1]{MR3959431} for $\wtu_{n+1}$, we obtain
\begin{equation*}
|u|_{\wsp} \lesssim \sum\limits_{i=1}^{n+1} \|\wtu_i\|_{\mathcal{X}^{s,p}(U_i)}.
\end{equation*}
By the definition of $\wtu_i$ and Lemmas \ref{lem:unit} and \ref{lem:0ext}, we get that for every $i=1,\ldots,n+1$, $$\|\wtu_i\|_{\mathcal{X}^{s,p}(U_i)} \lesssim \|R_i(u_i)(R_i^{-1}(\cdot))\|_{\mathcal{X}^{s,p}(R_i(B_i\cap D))} = \|u_i\|_{\mathcal{X}^{s,p}(B_i\cap D)}\lesssim \|u\|_{\ssp}.$$
This concludes the proof of the Korn's inequality for $\wosp$. The result for $\sosp$ is obtained as in the last part of the proof of Theorem \ref{th:goodlip}.
\end{proof}
\section{Application of the Hardy's inequality}\label{sec:Hardy}
In \cite[Theorem 2.3]{MR4017780} Mengesha gives a Hardy-type inequality for the half-space $\mR^d_+$, $p\geq 1$, $s\in (0,1)$, $sp\neq 1$, and $u\in C^1_c(\mR^d_+)$:
\begin{equation*}
\int_{\mR^d_+}\frac{|u(x)|^p}{x_d^{sp}}\, dx \lesssim |u|_{\mathcal{X}^{s,p}(\mR^d_+)}^p.
\end{equation*}
In this section we give a simple framework which allows to obtain the Korn's inequality for open sets $D$ directly from the Korn's inequality for the whole space and the Hardy's inequality for $D$.
\begin{proposition}
	Let $p>1$, $s\in(0,1)$. Assume that the open set $D\subset \mR^d$ admits the following Hardy's inequality for $u\in (C^1_c(D))^d$:
	\begin{equation*}
	\int_D \frac{|u(x)|^p}{d(x,D^c)^{sp}}\, dx \lesssim |u|_{\ssp}^p.
	\end{equation*}
	 Then the Korn's inequality holds for $D$, that is, there exists $C\geq 1$ such that for $u\in (C^1_c(D))^d$,
	 \begin{equation*}
	  C|u|_{\ssp}\geq |u|_{\wsp} .
	 \end{equation*}
\end{proposition}
\begin{proof}
	Let $u\in (C^1_c(D))^d$ and let $\wtu$ be the vector field $u$ extended to the whole of $\mR^d$ by $0$. First, by the Korn's inequality for the whole space \cite[Theorem 1.1]{MR3959431} we obtain
	\begin{equation*}\label{eq:fullkorn}
	|u|_{\wsp}\leq |\wtu|_{W^{s,p}(\mR^d)} \lesssim |\wtu|_{\mathcal{X}^{s,p}(\mR^d)}.
	\end{equation*}
	We estimate the right-hand side as follows:
	\begin{align*}\label{eq:ssprd}
	|\wtu|_{\mathcal{X}^{s,p}(\mR^d)}^p \leq |u|_{\ssp}^p + 2\int_D|u(x)|^p\int_{D^c}\frac{dy}{|x-y|^{d+sp}}\, dx.
	\end{align*}
	By using the polar coordinates we see that for every $x\in D$, $$\int_{D^c} |x-y|^{-d-sp}\, dy \leq \int_{B(0,d(x,D^c))^c} |y|^{-d-sp}\, dy \lesssim d(x,D^c)^{-sp}.$$
	Therefore, by the Hardy's inequality we get
	$$\int_D|u(x)|^p\int_{D^c}\frac{dy}{|x-y|^{d+sp}}\, dx \lesssim \int_D \frac{|u(x)|^p}{d(x,D^c)^{sp}}\, dx \lesssim |u|_{\ssp}^p,$$
	which ends the proof.
\end{proof}
\begin{remark}
 Thanks to the above result, we can significantly simplify the proof of the Korn's inequality for the half-space by Mengesha by omitting the discussion of the extension operator \cite[Section 4.1]{MR4017780}. If we had at our disposal the Hardy's inequality for the sets discussed in Theorem \ref{th:main}, we would obtain a slightly stronger inequality: $|u|_{\wsp} \lesssim |u|_{\ssp}$, that is, the estimate without the $L^p$ norm of $u$ on the right-hand side.
\end{remark}

\bibliographystyle{abbrv}
\bibliography{Korn}

\end{document}